\documentclass[12pt]{amsart}
\usepackage{amscd,amsmath,amsthm,amssymb,graphics}
\usepackage[dvips]{graphicx}
\usepackage{multicol}
\usepackage{epic,eepic}
\usepackage{amsfonts,amssymb,amscd,amsmath,enumitem,verbatim}

\def\NZQ{\Bbb}               

\def\KK{{\NZQ K}}
\def\frk{\frak}               

\def\Phi{{\frk n}}
\def\Phi{{\frk N}}
\def\fb{{\bold f}}
\def\opn#1#2{\def#1{\operatorname{#2}}} 
\opn\chara{char} \opn\length{\ell} \opn\pd{pd} \opn\rk{rk}
\opn\projdim{proj\,dim} \opn\injdim{inj\,dim} \opn\rank{rank}
\opn\depth{depth} \opn\grade{grade} \opn\height{height}
\opn\embdim{emb\,dim} \opn\codim{codim}

\opn\Tr{Tr} \opn\bigrank{big\,rank}
\opn\superheight{superheight}\opn\lcm{lcm}
\opn\trdeg{tr\,deg}
\opn\reg{reg} \opn\lreg{lreg} \opn\ini{in} \opn\lpd{lpd}
\opn\size{size}\opn\bigsize{bigsize}
\opn\cosize{cosize}\opn\bigcosize{bigcosize}
\opn\sdepth{sdepth}\opn\sreg{sreg}
\opn\link{link}\opn\fdepth{fdepth}\opn\m{m}
\opn\div{div} \opn\Div{Div} \opn\cl{cl} \opn\Cl{Cl}
\let\epsilon\varepsilon
\let\phi=\varphi
\let\kappa=\varkappa
\opn\Spec{Spec} \opn\Supp{Supp} \opn\supp{supp} \opn\Sing{Sing}
\opn\Ass{Ass} \opn\Min{Min}\opn\Mon{Mon} \opn\dstab{dstab} \opn\astab{astab}
\opn\Syz{Syz}
\opn\Ann{Ann} \opn\Rad{Rad} \opn\Soc{Soc}
\opn\Im{Im} \opn\Ker{Ker} \opn\Coker{Coker} \opn\Am{Am}
\opn\Hom{Hom} \opn\Tor{Tor} \opn\Ext{Ext} \opn\End{End}
\opn\Aut{Aut} \opn\id{id}

\opn\nat{nat}
\opn\pff{pf}
\opn\Pf{Pf} \opn\GL{GL} \opn\SL{SL} \opn\mod{mod} \opn\ord{ord}
\opn\Gin{Gin} \opn\Hilb{Hilb}\opn\sort{sort}
\opn\initial{init}
\opn\ende{end}
\opn\height{height}
\opn\type{type}
\opn\set{set}
\opn\aff{aff} \opn\con{conv} \opn\relint{relint} \opn\st{st}
\opn\lk{lk} \opn\cn{cn} \opn\core{core} \opn\vol{vol}
\opn\link{link} \opn\star{star}\opn\lex{lex} \opn\tlex{t-lex} \opn\shad{Shad}

\opn\sqlex{sqlex}
\opn\gr{gr}

\def\pot#1#2{#1[\kern-0.28ex[#2]\kern-0.28ex]}

\opn\dirlim{\underrightarrow{\lim}}
\opn\inivlim{\underleftarrow{\lim}}
\def\Implies{\ifmmode\Longrightarrow \else
        \unskip${}\Longrightarrow{}$\ignorespaces\fi}
\def\implies{\ifmmode\Rightarrow \else
        \unskip${}\Rightarrow{}$\ignorespaces\fi}
\def\iff{\ifmmode\Longleftrightarrow \else
        \unskip${}\Longleftrightarrow{}$\ignorespaces\fi}

\let\:=\colon
 \theoremstyle{plain}
\newtheorem{Theorem}{Theorem}[section]
 \newtheorem{Lemma}[Theorem]{Lemma}
 
 \newtheorem{Proposition}[Theorem]{Proposition}

 \theoremstyle{definition}
 \newtheorem{Definition}[Theorem]{Definition}
 \newtheorem{Remark}[Theorem]{Remark}
 
 \newtheorem{Example}[Theorem]{Example}

\let\epsilon\varepsilon
\let\kappa=\varkappa
\opn\dis{dis}
\def\pnt{{\raise0.5mm\hbox{\large\bf.}}}

\opn\Lex{Lex}

\begin{document}
\title{Kruskal-Katona Theorem for $T$-spread strongly stable ideals}
\author {Claudia Andrei-Ciobanu}

\address{Faculty of Mathematics and Computer Science, University of Bucharest,  Str. Academiei 14, Bucharest – 010014, Romania, \emph{E-mail address: claudiaandrei1992@gmail.com}}

\thanks{This paper was written while the author visited the Department of Mathematics of the University of Duisburg-Essen. The author gratefully acknowledges the financial support awarded by the European Mathematical Society and the Doctoral School in Mathematics of the University of Bucharest.}

\begin{abstract}
We prove that any $t$-spread strongly stable ideal has a unique $t$-spread lex ideal with the same $f$-vector. We also characterize the possible $f$-vectors of $t$-spread strongly stable ideals in the "$t$-spread" analogue of Kruskal-Katona theorem.
\end{abstract}

\thanks{The author would also like to thank Professor J\"urgen Herzog for valuable suggestions and comments during the preparation of this paper.}

\maketitle
\section*{Introduction}

Kruskal-Katona Theorem solves the problem of characterizing the possible $f$-vectors of the simplicial complexes on a given vertex set. With a special importance in Combinatorial Algebra, this result has been proved by Joseph Kruskal \cite{Krusk} and Gyula Katona \cite{Kat}. In other words, Kruskal-Katona Theorem gives an elegant answer for the following question: \emph{Given the number of faces of dimension $d-1$ of a simplicial complex $\Delta$, how many faces of dimension $d$ could the complex have?}.

Let $\Delta$ be a $(d-1)$-dimensional simplicial complex on the vertex set $[n]$. The $f$-vector of $\Delta$ is defined as a sequence of positive integers $f=(f_{-1}, f_0, f_1, \ldots, f_{d-1})$ with the property that $f_{-1}=1$ and $f_i$ counts the faces of $\Delta$ of dimension equal to $i$ for $i\in\{0,1,\ldots, d-1\}$. Then Kruskal-Katona Theorem claims that a sequence $f=(f_{-1}, f_0, f_1, \ldots, f_{d-1})\in\mathbb{Z}^{d+1}$ is the $f$-vector of some $(d-1)$-dimensional simplicial complex if and only if $f_{-1}=1$ and $0<f_{i+1}\leq f_{i}^{(i+1)}$ for all $i$, where $f_i^{(i+1)}$ is determined by the so-called binomial or Macaulay expansion. More precisely, given two integers $a, d>0$, let
\[a={a_d \choose d}+{{a_{d-1}} \choose {d-1}}+\cdots + {a_r \choose r}\]
with $a_d>a_{d-1}>\ldots>a_r\geq r\geq 1$ be the unique binomial expansion of $a$ with respect to $d$. Then
\[a^{(d)}={a_d \choose {d+1}}+{{a_{d-1}} \choose {d}}+\cdots + {a_r \choose {r+1}}.\]

An algebraic proof of Kruskal-Katona Theorem can be found in \cite{HHBook}. Let $\Delta$ be a $(d-1)$-dimensional simplicial complex on the vertex set $[n]$, $\KK$ be a field and $\KK\{\Delta\}$ be the exterior face ring of $\Delta$. Then $\KK\{\Delta\}=E/J_{\Delta}$, where $E$ is the exterior algebra of the $\KK$-vector space $V=\oplus_{i=1}^{n}\KK e_i$ and $J_{\Delta}\subset E$ is the graded ideal generated by all exterior monomials $e_{F}=e_{i_1}\wedge \ldots\wedge e_{i_k}$ for which $F=\{i_1,\ldots, i_k\}\notin \Delta$. Since the Hilbert series of $\KK\{\Delta\}$ is $H_{\KK\{\Delta\}}(t)=\sum_{i=0}^{d}f_{i-1}t^i$, where $(f_{-1}, f_0, f_1, \ldots, f_{d-1})$ is the $f$-vector of $\Delta$, in order to get the conditions on the $f_i$'s, we need to obtain the possible Hilbert functions of graded algebras of the form $E/J$. The main step of the proof consists of showing that for each graded ideal $J\subset E$, there exists a unique square-free lexsegment ideal $J^{\sqlex}\subset E$ such that $H_{E/J}(t)=H_{E/J^{\sqlex}}(t)$. Therefore, it remains to understand the Hilbert series of a square-free lexsegment ideal.

In this paper, the key role is played by $t$-spread strongly stable ideals with $t\geq1$. They have been recently introduced in \cite{EHQ} and they represent a special class of square-free monomial ideals. Let $t\geq 1$ be a positive integer, $\KK$ be a field and $S=\KK[x_1, \ldots, x_n]$ be the polynomial ring in $n$ variables over $\KK$.

A monomial $x_{i_1}x_{i_2}\cdots x_{i_d}\in S$ with $i_1\leq i_2\leq\cdots \leq i_d$  is called $t$-spread, if $i_j- i_{j-1}\geq t$ for $2\leq j \leq n$ and a  monomial ideal in $S$ is called  a $t$-spread monomial ideal,  if it is generated by  $t$-spread monomials. Note that every square-free monomial is a $1$-spread monomial and thus, every square-free monomial ideal is a $1$-spread ideal.

Let $I$ be a $t$-spread ideal in $S$ and let $\Delta$ be its associated simplicial complex. It is natural to introduce the $f_t$-vector of $I$, $${\fb}_{t}(I)=(f_{t,-1}(I),f_{t,0}(I)\ldots,f_{t,j}(I), \ldots),$$ where $f_{t,j}(I)$ is the cardinality of the set
\[\{F\;:\;F\text{ is a }j-\text{dimensional face in }\Delta\text{ and }x_F=\prod_{i\in F}x_i\text{ is a }t-\text{spread monomial}\}.\]

Since ${\fb}_1(I)$ is the classical $f$-vector of $\Delta$, a natural question is the following: \emph{Are there any results for $f_t(I)$ which generalize the classification of the classical $f$-vectors for $t$-spread strongly stable ideals with $t\geq 1$?}. A complete answer for $t$-spread strongly stable ideals is given in the main result of this paper,
\\ {\bf Theorem~\ref{Kruskalkatona}.}
Let $f=(f(0),f(1),\ldots, f(d),\ldots)$ be a sequence of positive integers and $t\geq 1$ be an integer. The following conditions are equivalent:
  \begin{enumerate}
    \item[\emph{(1)}] there exists an integer $n\geq 0$ and a $t$-spread strongly stable ideal $$I\subset \KK[x_1,\ldots,x_n]$$ such that $f(d)=f_{t,d-1}(I)$ for all $d$.
    \item[\emph{(2)}] $f(0)=1$ and $f(d+1)\leq f(d)^{[d]_t}$ for all $d\geq 1$.
  \end{enumerate}
The $t$-operator $f(d)\rightarrow f(d)^{[d]_t}$ is determined analogously to the operator which is involved in the proof of Kruskal-Katona Theorem; see Definition~\ref{power} and Remark~\ref{clasica}.

The proof of this theorem follows the steps in the proof of Kruskal-Katona Theorem given in \cite{HHBook}.
Thus, in Section 1, we define the $t$-spread lex ideal and we recall from \cite{EHQ} what a $t$-spread strongly stable ideal means.
In Section 2, Theorem~\ref{exists} shows that for any $t$-spread strongly stable ideal there exists a unique $t$-spread lex ideal with the same $f_t$-vector.
By Remark~\ref{tgeq2}, a $t$-spread ideal may not have an associated $t$-spread lex ideal with the same $f_t$-vector. That is why we will restrict to the case when the ideal is $t$-spread strongly stable. Classifying all $t$-spread ideals which have an associated $t$-spread lex ideal with the same $f_t$-vector remains still open. In Section 3, we present a complete classification of the sequences of positive integers which are the $f_t$-vectors of some $t$-spread strongly stable ideals; see Theorem~\ref{Kruskalkatona}.

\section{Preliminaries}

Fix a field $\KK$ and a polynomial ring $S=\KK[x_1,\ldots, x_n]$.
A monomial $x_{i_1}x_{i_2}\cdots x_{i_d}\in S$ with $i_1\leq i_2\leq\cdots \leq i_d$  is called $t$-spread, if $i_j- i_{j-1}\geq t$ for $2\leq j \leq n$. Note that any monomial is $0$-spread, while the square-free monomials are $1$-spread.

A  monomial ideal in $S$ is called  a $t$-spread monomial ideal,  if it is generated by  $t$-spread monomials. For example,  $I=(x_1x_4x_8,x_2x_5x_8,x_1x_5x_9,x_2x_6x_9,x_4x_9) \subset K[x_1, \ldots, x_9]$ is a $3$-spread monomial ideal, but not $4$-spread, because $x_2x_5x_8$ is not a $4$-spread monomial.

For an arbitrary monomial ideal $I$, we denote by $I_j$, the $j$-th graded component of $I$ and call the set of $t$-spread monomials in $I_j$, the $t$-spread part of $I_j$ and denote it by ${[I_j]}_t$. Furthermore, we set
\[f_{t,j-1}(I)=|{[S_j]}_t |-|{[I_j]}_t|.\]

Then the vector ${\fb}_{t}(I)=(f_{t,-1}(I),f_{t,0}(I)\ldots,f_{t,j}(I), \ldots)$ is called the \emph{$f_t$-vector} of the $t$-spread monomial ideal $I$. By convention, we set $f_{t,-1}=1$.
Note that if $t=1$ then $I$ is the Stanley-Reisner ideal  of a uniquely determined simplicial complex $\Delta$ and ${\fb}_1(I)$ is the classical $f$-vector of $\Delta$. According to \cite[Theorem 5.1.7]{BH98}, the $f_1$-vector of 1-spread monomial ideal $I\subset S$ determines the Hilbert function of $S/I$. This is not the case for $t\geq 2$. For example, $I_1=(x_1x_3,x_2x_4)$ and $I_2=(x_1x_3,x_1x_4)$ are $2$-spread monomial ideals in $\KK[x_1,\ldots, x_4]$ with ${\fb}_2(I_1)={\fb}_2(I_2)=(1,4,1,0,0,\ldots)$ and $H(S/I_1,3)=12<13=H(S/I_2,3)$.

We denote by $M_{n,d,t}$ the set of the $t$-spread monomials of degree $d$ in the polynomial ring $S$. For a monomial $u\in S$, we set \[\supp(u)=\{i:x_i\mid u\}\text{ and }\m(u)=\max\{i: i\in \supp(u)\}.\]

\begin{Definition}

\begin{itemize}
  \item[(a)] A subset $L\subset M_{n,d,t}$ is called a {\em $t$-spread strongly stable set}, if for all $t$-spread monomials $u\in L$, all $j\in\supp(u)$ and all $1\leq i<j$ such that $x_i(u/x_{j})$ is a $t$-spread monomial, it follows that $x_i(u/x_j)\in L$.
  \item[(b)] Let $I$ be a $t$-spread monomial ideal. Then $I$ is called a {\em $t$-spread strongly stable ideal}, if ${[I_j]}_t$ is a $t$-spread strongly stable set for all $j$.
\end{itemize}
\end{Definition}

A special class of $t$-spread strongly stable ideals consists of $t$-spread lex ideals, which are defined as follows.

\begin{Definition}

\begin{itemize}
  \item[(a)] A subset $L\subset M_{n,d,t}$ is called a {\em $t$-spread lex set}, if for all $u \in L$ and for all $v \in M_{n,d,t}$ with $v >_{\lex} u$, it follows that $v \in L$.
  \item[(b)] Let $I$ be a $t$-spread monomial ideal. Then $I$ is called a {\em $t$-spread lex ideal}, if ${[I_j]}_t$ is a $t$-spread lex set for all $j$.
\end{itemize}
\end{Definition}

Let $L\subset M_{n,d,t}$ be a $t$-spread lex set.  Note that $L$ need not to be a $t$-spread lex set in $M_{m,d,t}$ for $m>n$. For example,  $L=\{x_1x_2, x_1x_3, x_2x_3\}$ is  a 1-lex set in $M_{3,2,1}$, but not in $M_{4,2,1}$.  However, if $L$ is a $t$-spread strongly stable set in $M_{n,d,t}$, then $L$ remains a $t$-spread strongly stable set in $M_{m,d,t}$ for all $m>n$.

We notice that $\{I\subset S:I\text{ is a }t-\text{spread lex ideal}\}\subset\{I\subset S:I\text{ is a }t-\text{spread strongly stable ideal}\}$ and the inclusion is strict according to the following example.

\begin{Example}
Let $I=(x_1x_3, x_1x_4, x_2x_4x_6, x_2x_4x_7)\subset \KK[x_1,\ldots, x_7]$. Then $I$ is a $2$-spread strongly stable ideal. Since $x_1x_5x_7>_{\lex}x_2x_4x_6$ and $x_1x_5x_7\notin I$, the ideal $I$ is not a $2$-spread lex ideal.
\end{Example}

For every $L \subset M_{n,d,t}$ and for every $0\leq\tau\leq t$, we define the $\tau$-shadow of $L$
\[\shad_{\tau}(L)=\{x_iv\; : \; v \in L , 1\leq i\leq n \text{ and } x_iv \text{ is a $\tau$-spread monomial} \}\]
\begin{Lemma}\label{shad}
\begin{itemize}
\item[(a)] Let $L \subset M_{n,d,t}$ be a $t$-spread strongly stable set.

Then $\shad_t(L) \subset M_{n,d+1,t}$ is also a $t$-spread strongly stable set.
\item[(b)] Let $L \subset M_{n,d,t}$ be a $t$-spread lex set.

Then $\shad_t(L)\subset M_{n,d+1,t}$ is also a $t$-spread lex set.
\end{itemize}
\end{Lemma}
\begin{proof}
Let $u\in\shad_t(L)$. Then $u =wx_j$ for some $w\in L$. We may assume that $\m(w)\leq j$. Otherwise, we consider $u=w'x_{\m(u)}$, where $w'=x_j(w/x_{\m(u)})\in L$.
\begin{itemize}
  \item[(a)] Let $v=x_i(u/x_k)=x_i(wx_j)/x_k$ be a $t$-spread monomial such that $k \in \supp(u)$ and $i<k$. Then we need to show that $v\in \shad_t(L)$. If $k=j$, then $v=wx_i\in\shad_t(L)$, by definition of $\shad_t(L)$. If $k\neq j$, then $x_iw/x_j\in L$ and again $v\in \shad_t(L)$.
  \item[(b)] Let $v\in M_{n,d+1,t}$ with $v >_{\lex} u$. Then we need to show that $v \in \shad_t(L)$. Let $v=x_{i_1}\cdots x_{i_{d+1}}$ with $i_1 \leq \ldots \leq i_{d+1}$ and $u=x_{k_1}\cdots x_{k_{d}}x_j$ with $k_1 \leq \ldots \leq k_{d} \leq j$. Since $v >_{\lex} u$, $v'=x_{i_1}\cdots x_{i_{d}} \geq_{\lex} w=x_{k_1}\cdots x_{k_{d}}$. This shows that $v' \in L$ and therefore $v\in\shad_t(L)$ because $v=v'x_{i_{d+1}}$.
\end{itemize}
\end{proof}

\begin{Remark}
The assertions of the previous lemma do not remain true for every $0\leq \tau<t$. Indeed, let $L=\{x_1x_3,x_1x_4,x_1x_5,x_2x_4\}\in M_{5,2,2}$ be a $2$-spread lex set. Then $\shad_1(L)=\{x_1x_2x_3, x_1x_2x_4, x_1x_2x_5, x_1x_3x_4, x_1x_3x_5, x_1x_4x_5, x_2x_3x_4, x_2x_4x_5\}$ is not a $1$-spread strongly stable set because $x_2x_3x_5\notin \shad_1(L)$.
\end{Remark}

\section{The existence of $I^{\tlex}$}

Let $I \subset S$ be a $t$-spread strongly stable monomial ideal. Then a $t$-spread lex ideal $J \subset S$ with $\fb_t(I)= \fb_t(J)$, if exists, is uniquely determined.  We then denote this ideal $J$ by $I^{\tlex}$. The main purpose of this section is to prove

\begin{Theorem}\label{exists}
For any $t$-spread strongly stable ideal $I$, the $t$-spread lex ideal $I^{\tlex}$ exists.
\end{Theorem}

For the proof, we proceed in a similar way to the proof of the existence of $I^{\lex}$ and $I^{\sqlex}$; see \cite[Chapter 6]{HHBook}.

For each graded component $I_j$ of $I$, let $I_j^{\tlex}$ be the $\KK$-vector space spanned by $L_j\cup \shad_0(B_{j-1})$, where $L_j$ is the unique $t$-spread lex set with $|L_j|=|{[I_j]}_t|$ and where $B_{j-1}$ is the set of the monomials of $I_{j-1}^{\tlex}$. If $j=0$, then we consider $B_{j-1}=\emptyset$. We define $I^{\tlex}=\bigoplus_j I_j^{\tlex}$.
This is the only possible candidate meeting the requirements of the theorem if and only if $\shad_t(L_j)\subset L_{j+1}$ for all $j$. Indeed, $I^{\tlex}$ is an ideal in $S$ because $B_j$ is a $\KK$-basis for $I_j^{\tlex}$ and $\shad_0(B_{j-1})\subset B_j=L_{j}\cup \shad_0(B_{j-1})$ for every $j\geq 0$.

Let $d$ be the smallest degree which appears in the set of the generators of $I$. We notice that $B_j=\emptyset=L_j$ for all $j<d$ and $B_d=L_d$. Then \[\shad_t(B_j)\subset L_{j+1}\text{ and }|[I_j^{\tlex}]_t|=|L_j|=|[I_j]_t|\] if and only if \[\shad_t(L_j)\subset L_{j+1}\] for all $j\leq d$. By using induction on $j$,
\begin{align*}
  \shad_t(B_j)= & \shad_t(L_j\cup \shad_0(B_{j-1})) \\
   = & \shad_t(L_j)\cup\shad_t(\shad_0(B_{j-1}))\subset L_{j+1}\cup \shad_t(\shad_t(B_{j-1}))=L_{j+1}
\end{align*}
and \[|[I_{j+1}^{\tlex}]_t|=|L_{j+1}|=|[I_{j+1}]_t|\]
if and only if \[\shad_t(L_j)\subset L_{j+1}.\]

The proof of Theorem~\ref{exists} is completed if we show that $|\shad_t(L_j)|\leq |[I_{j+1}]_t|=|L_{j+1}|$, since $L_j$ is a $t$-spread lex set for every $j\geq 0$.

\medskip

Let $L\subset M_{n,d,t}$ be a set of monomials. For every $i\in\{1+t(d-1),2+t(d-1)\ldots, n\}$, we denote by $m_i(L)$ the number of elements $u\in L$ with $m(u)=i$ and we set $m_{\leq i}(L)=\sum_{j=1}^{i}m_j(L)$. Then we have the following result.

\begin{Lemma}\label{shadow}
Let $L\subset M_{n,d,t}$ be a $t$-spread strongly stable set. Then

\begin{itemize}
\item[\emph{(a)}] $m_i(\shad_t(L))=m_{\leq{i-t}}(L)$ for all $i$ and
\item[\emph{(b)}] $|\shad_t(L)|=\sum_{i=1+(d-1)t}^{n-t}m_{\leq i}(L)$.
\end{itemize}
\end{Lemma}

\begin{proof}
  $(b)$ is a consequence of $(a)$. For the proof of $(a)$, we note that the map
  \[\phi:\{u\in L: m(u)\leq i-t\}\rightarrow \{u\in \shad_t(L): m(u)=i\}, u\rightarrow ux_i\]
  is a bijection. In fact, $\phi$ is clearly injective. To see that $\phi$ is surjective, let $v\in \shad_t(L)$ with $m(v)=i$. Since $v\in \shad_t(L)$, there exists $u\in L$ with $v=ux_j$ for some $j\leq i$. If $j=i$, then $m(u)<m(v)$ and $i-m(u)=m(v)-m(u)\geq t$ because $v$ is a $t$-spread monomial. In other words, $u\in \{u\in L: m(u)\leq i-t\}$ and $\phi(u)=v$. Otherwise, $j<i$ and $i\in \supp(u)$. Since $L$ is a strongly stable set and $v$ is a $t$-spread monomial, $w=x_j(u/x_i)\in L$. Thus, $v=wx_i$ and $\phi(w)=v$.
\end{proof}

Now Theorem~\ref{exists} will be a consequence of the following theorem.
\begin{Theorem}\label{Bayer}
Let $L\subset M_{n,d,t}$ be a $t$-spread lex set and let $N\subset M_{n,d,t}$ be a $t$-spread strongly stable set with $|L|\leq |N|$. Then $m_{\leq i}(L)\leq m_{\leq i}(N)$.
\end{Theorem}
\begin{proof}
  We first observe that $N=N_0\cup N_1x_n$ where $N_0$ is a $t$-spread strongly stable set of monomials of degree $d$ in the variables $x_1, x_2,\ldots, x_{n-1}$ and $N_1$ is a $t$-spread strongly stable set of monomials of degree $d-1$ in the variables $x_1,\ldots, x_{n-t}$. Similarly, one has the decomposition $L=L_0\cup L_1x_n$ where $L_0$ and $L_1$ are $t$-spread lex sets.

  We prove the theorem by induction on the number of variables. For $n=1$, the assertion is trivial. Now let $n>1$. Since $|L|=m_{\leq n}(L)$ and $|N|=m_{\leq n}(N)$, we obtain $m_{\leq n}(L)\leq m_{\leq n}(N)$.

  Note that $|L_0|=m_{\leq n-1}(L)$ and $|N_0|=m_{\leq n-1}(N)$. Thus in order to prove that $m_{\leq n-1}(L)\leq m_{\leq n-1}(N)$, we need to show that $|L_0|\leq |N_0|$. Moreover, if the inequality $|L_0|\leq |N_0|$ holds, then by applying the induction hypothesis we obtain \[m_{\leq i}(L)=m_{\leq i}(L_0)\leq m_{\leq i}(N_0)=m_{\leq i}(N)\] for every $i\in \{1,\ldots, n-1\}$.

  Let $N_0^*\subset M_{n-1, d, t}$ be a $t$-spread lex set with $|N_0^*|=|N_0|$ and $N_1^*\subset M_{n-t, d-1, t}$ be a $t$-spread lex set with $|N_1^*|=|N_1|$. We claim that $N^*=N_0^*\cup N_1^*x_n$ is again a $t$-spread strongly stable set of monomials. Indeed, we need to show that $\shad_t(N_1^*)\subset N_0^*$. By Lemma~\ref{shad}, it is clear that $\shad_t(N_1^*)$ is a $t$-spread lex set. Thus, it suffices to prove that $|\shad_t(N_1^*)|\leq |N_0^*|$. Since $N$ is $t$-spread strongly stable set, $\shad_t (N_1)\subset N_0$. We apply Lemma~\ref{shadow} and our induction hypothesis and we obtain
  \begin{align*}
    |\shad_t(N_1^*)|= &  \sum_{i=1+(d-1)t}^{n-t}m_{\leq i}(N_1^*)\\
     \leq  & \sum_{i=1+(d-1)t}^{n-t}m_{\leq i}(N_1)=|\shad_t(N_1)|\leq|N_0|=|N_0^*| \\
     &
  \end{align*}
  Thus, we completed the proof of the fact $N^*$ is a $t$-spread strongly stable set of monomials.

  Since $|N|=|N^*|$, we may replace $N$ by $N^*$ and then we can assume that $N_0$ is a $t$-spread lex set. We suppose that $n\neq (d-1)t+1$. Otherwise, $M_{n,d,t}=\{x_1x_{1+t}\cdots x_{1+(d-1)t} \}$ and the assertion is trivial.

  Let $m=x_{j_1}\cdots x_{j_d}$ be a $t$-spread monomial and $\alpha: M_{n,d,t}\rightarrow M_{n,d,t}$ be a map defined as follows:
  \begin{enumerate}
    \item if $j_d\neq n$, then $\alpha(m)=m$.
    \item if $j_d=n$ and there exists $r\in\{2,\ldots, d\}$ such that $j_r>j_{r-1}+t$, then we choose $r$ to be the largest integer with this property and define \[\alpha(m)=x_{j_1}\cdots x_{j_{r-1}}x_{j_r-1}\cdots x_{j_{d-1}-1}x_{n-1}.\]
    \item if $m=x_{n-(d-1)t}x_{n-(d-2)t}\cdots x_{n}$, then $\alpha(m)=x_{n-1-(d-1)t}x_{n-1-(d-2)t}\cdots x_{n-1}$.
  \end{enumerate}
  Then $\alpha$ is a lexicographic order preserving map. Indeed, if we take $m_1=x_{j_1}x_{j_2}\cdots x_{j_d}$ and $m_2=x_{k_1}x_{k_2}\cdots x_{k_d}$ with $m_1<_{lex}m_2$, then $\alpha(m_1)<_{lex}\alpha(m_2)$ in every case. For example, if $m_1\in M_{n-1,d,t}$ and $m_2$ is as in case $(2)$, then there exists $l\in \{1,\ldots, d\}$ such that $j_1=k_1, \ldots, j_{l-1}=k_{l-1}$ and $j_l>k_l$ and for
  \begin{itemize}
                        \item[(a)] $r>l$, we have $\alpha(m_1)<_{lex}\alpha(m_2)$, since $j_1=k_1, \ldots, j_{l-1}=k_{l-1}$ and $j_l>k_l$.
                        \item[(b)] $r<l$, we have $\alpha(m_1)<_{lex}\alpha(m_2)$, since $j_1=k_1, \ldots, j_{r-1}=k_{r-1}$ and $j_r=k_r>k_r-1$.
                        \item[(c)] $r=l$, we have $\alpha(m_1)<_{lex}\alpha(m_2)$, since $j_1=k_1, \ldots, j_{l-1}=k_{l-1}$ and $j_l>k_l>k_l-1$.
  \end{itemize}
  All the other cases can be treated in the same way.

  For a set of monomials $\mathcal{S}$ we denote by $\min \mathcal{S}$ the lexicographically smallest element in $\mathcal{S}$. Since both $L_0$ and $N_0$ are $t$-spread lex sets, the inequality  $|L_0|\leq |N_0|$ will follow once we have shown that $\min L_0\geq \min N_0$. Let $u=x_{k_1}\cdots x_{k_d}=\min L$ and $v=x_{j_1}\cdots x_{j_d}=\min N$. Then $\alpha(u)=\min L_0$ and $\alpha(v)=\min N_0$. Indeed, according to the three cases which define the map $\alpha$, we have:
  \begin{enumerate}
    \item $v\in N_0$ and $\alpha(v)=v\in N_0$.
    \item $v\in N_1x_n$ and $\alpha(v)=x_{j_1}\cdots x_{j_{k-1}}x_{j_k-1}\cdots x_{j_{d-1}-1}x_{n-1}$ with $k=\max\{r\;:\; j_r>j_{r-1}+t\text{ and }2\leq r\leq d\}$.

        If $k=d$, then $\alpha(v)=x_{j_1}\cdots x_{j_{d-1}}x_{n-1}=(v/x_n)x_{n-1}\in N_0$ because $N$ is a $t$-spread strongly stable set.

        If $k<d$, then we set $v_1=x_{j_k-1}(v/x_{j_k})\in N$, $v_2=x_{j_{k+1}-1}(v_1/x_{j_{k+1}})\in N$, $\ldots, v_{d-k}=x_{j_{d-1}-1}(v_{d-k-1}/x_{j_{d-1}})\in N$ and $\alpha(v)=x_{n-1}(v_{d-k}/x_n)\in N_0$, since $N$ is a $t$-spread strongly stable set.
    \item $v\in N_1x_n$ and $\alpha(v)=x_{j_1-1}\cdots x_{j_{d-1}-1}x_{n-1}$ with $j_r=j_{r-1}+t$ for all $2\leq r\leq d$. Since $N$ is a $t$-spread strongly stable set, $v_1=x_{j_1-1}(v/x_{j_1})\in N$, $v_2=x_{j_2-1}(v_1/x_{j_2})\in N$, $\ldots, v_{d-1}=x_{j_{d-1}-1}(v_{d-2}/x_{j_{d-1}})\in N$ and $\alpha(v)=x_{n-1}(v_{d-1}/x_n)\in N_0$.
  \end{enumerate}
 Then $\min N_0\leq \alpha(v)$ and $\min N_0\geq v=\min N$. Thus, $\alpha(\min N_0)=\min N_0\geq \alpha (v)$ and $\min N_0=\alpha(v)$. Similarly, we obtain $\min L_0=\alpha(u)$.

 Finally we observe that $u\geq v$, since $L$ is a $t$-spread lex set and $|L|\leq |N|$. Hence we conclude that $\min L_0=\alpha(u)\geq \alpha(v)=\min N_0$, as desired.
\end{proof}
\begin{Example}\label{imp}
Let
\begin{align*}
  I = &(x_1x_3x_5, x_1x_3x_6, x_1x_3x_7, x_1x_3x_8, x_1x_4x_6, x_1x_4x_7, x_1x_4x_8, \\
    & x_2x_4x_6, x_2x_4x_7, x_2x_4x_8)\subset \KK[x_1,\ldots, x_8].
\end{align*}
Then $I$ is a $2$-spread strongly stable ideal in $\KK[x_1,\ldots, x_8]$ and
\begin{align*}
  I^{\tlex}=&(x_1x_3x_5,x_1x_3x_6,x_1x_3x_7,x_1x_3x_8,x_1x_4x_6,x_1x_4x_7,x_1x_4x_8,  \\
   & x_1x_5x_7,x_1x_5x_8,x_1x_6x_8,x_2x_4x_6x_8).
\end{align*}
Indeed, we have
\begin{align*}
  B_3=&L_3= \{x_1x_3x_5,x_1x_3x_6,x_1x_3x_7,x_1x_3x_8,x_1x_4x_6,x_1x_4x_7, \\
   & x_1x_4x_8,x_1x_5x_7,x_1x_5x_8,x_1x_6x_8\} \\
  B_4=&L_4\cup \shad_0(B_3)= \{x_1x_3x_5x_7,x_1x_3x_5x_8, x_1x_3x_6x_8,x_1x_4x_6x_8,\\
   &x_2x_4x_6x_8\}\cup \shad_0(B_3)\text{ and }\\
   B_j = & \shad_0(B_{j-1})\text{ for all }j\geq 5,
\end{align*}
since
\begin{align*}
  [I_3]_2 = & \{x_1x_3x_5, x_1x_3x_6, x_1x_3x_7, x_1x_3x_8, x_1x_4x_6, x_1x_4x_7, x_1x_4x_8,\\
   & x_2x_4x_6, x_2x_4x_7,x_2x_4x_8\}, \\
  [I_4]_2 = & \{x_1x_3x_5x_7,x_1x_3x_5x_8,x_1x_3x_6x_8,x_1x_4x_6x_8,x_2x_4x_6x_8\}\text{ and }\\
  [I_j]_2= &  \emptyset\text{ for all } j\geq 5.
\end{align*}

\end{Example}
\begin{Remark}\label{tgeq2}
In general, a $t$-spread monomial ideal may not have a $t$-spread lex ideal with the same $f_t$-vector.

For example, if $I=(x_2x_8, x_2x_6, x_2x_4)\subset \KK[x_1,\ldots x_8]$, then we have $B_2=L_2=\{x_1x_3, x_1x_4,x_1x_5\}$ and $|\shad_2(B_2)|=9>5=|[I_3]_2|$, which leads to the impossibility to construct the $2$-spread lex ideal with the same $f_2$-vector with $I$.

However, there exist $t$-spread monomial ideals which are not strongly stable, but we can construct for them a $t$-spread lex ideal with the same $f_t$-vector. For example, if we consider the $3$-spread monomial ideal, $I=(x_1x_7, x_2x_6, x_3x_6)\subset\KK[x_1,\ldots, x_7]$, which is not strongly stable because $x_1x_6\notin I$, then we obtain $I^{\tlex}=(x_1x_4, x_1x_5, x_1x_6)$.
\end{Remark}

\section{Possible $f$-vectors of a $t$-spread strongly stable ideal}

In this section, we will give a complete answer to the following question: When is a given sequence of positive integers $f_t=(f_{t,-1}, f_{t,0}, f_{t,1},\ldots, f_{t,d}, \ldots)$ the $f_t$-vector of a $t$-spread strongly stable ideal?

To answer this question, we would like to proceed like in the proof of Kruskal-Katona Theorem given in \cite{HHBook}. To this aim, we need to define a \emph{"$t$-operator"} analog to the operator $a\rightarrow a^{(d)}$ which is involved in the proof of Kruskal-Katona Theorem.
Let us recall the so-called binomial or Macaulay expansion of a positive integer.

\begin{Lemma}\label{expansion}\cite[Lemma 6.3.4]{HHBook}
Let $d$ be a positive integer. Then each positive integer $a$ has a unique expansion
\[a={a_d \choose d}+{{a_{d-1}} \choose {d-1}}+\cdots + {a_r \choose r}\]
with $a_d>a_{d-1}>\ldots>a_r\geq r\geq 1$. This expansion is called the binomial or Macaulay expansion of $a$ with respect to $d$.
\end{Lemma}

Let $a={a_d \choose d}+{{a_{d-1}} \choose {d-1}}+\cdots + {a_r \choose r}$ be the binomial expansion of $a$ with respect to $d$. Then one defines the binomial operator $a\rightarrow a^{(d)}$ by \[a^{(d)}={a_d \choose {d+1}}+{{a_{d-1}} \choose {d}}+\cdots + {a_r \choose {r+1}}.\] For convenience, $0^{(d)}=0$ for all positive integers $d$.

\begin{Theorem}\label{clasicKK}\cite[Theorem 6.4.5 (Kruskal-Katona)]{HHBook}
Let $f=(f_{-1},f_0,f_1, \ldots, f_{d-1})$ be a sequence of positive integers. Then the following conditions are equivalent:
\begin{enumerate}
  \item[\emph{(1)}] There exists a simplicial complex $\Delta$ with the $f$-vector of $\Delta$, $f(\Delta)=f$;
  \item[\emph{(2)}] $f_{-1}=1$ and $f_{j+1}\leq f_j^{(j+1)}$ for $0\leq j\leq d-2$.
\end{enumerate}
\end{Theorem}

Since every square-free monomial ideal has a square-free lexsegment ideal with the same $f_1$-vector, the proof of Theorem~\ref{clasicKK} given in \cite{HHBook} is available for all square-free monomial ideals. By Remark~\ref{tgeq2}, a $t$-spread ideal may not have an associated $t$-spread lex ideal with the same $f_t$-vector. Therefore, we will restrict to $t$-spread strongly stable ideals.

\bigskip
As in the ordinary square-free case, the binomial expansions naturally appear in the context of the $t$-spread lex sets. Let $u=x_{i_1}\cdots x_{i_d}\in M_{n,d,t}$. We denote by $L_u$ the $t$-spread lex set $\{v\in M_{n,d,t}: v\geq_{\lex}u\}$.
\begin{Lemma}\label{cardLu}
Let $u=x_{i_1}\cdots x_{i_d}\in M_{n,d,t}$. Then
\[M_{n,d,t}\setminus L_u = \bigcup_{k=1}^d \{v=x_{j_k}x_{j_{k+1}}\cdots x_{j_d}: v\in M_{n,d-k+1,t}, j_k>i_k\}\prod_{l=1}^{k-1}x_{i_l}.\]

This union is disjoint and in particular, we obtain
\[|M_{n,d,t}\setminus L_u|=\sum_{j=1}^{d}{a_j \choose j}\] with $a_j={n-i_{d-j+1}-(j-1)(t-1)}$ for every $j\in\{1,\ldots, d\}$.
\end{Lemma}
\begin{proof}
We will prove the equality between sets by induction on $d$. If $d=1$, then $M_{n,1,t}=\{x_{i_1+1},x_{i_1+2}, \ldots, x_{n}\}$, as desired.

If $d>1$, then
\[M_{n,d,t}\setminus L_u=\{v=x_{j_1}\cdots x_{j_d}: v\in M_{n,d,t}\text{ and }j_1>i_1\}\cup\]\[\cup (M_{n,d-1,t}\setminus L_{ux_{i_1}^{-1}})x_{i_1}\] and we may apply the induction hypothesis for $M_{n,d-1,t}\setminus L_{ux_{i_1}^{-1}}$.

In order to compute the cardinality of $M_{n,d,t}\setminus L_{u}$, we recall that $|M_{n,d,t}|={{n-(d-1)(t-1)}\choose d}$ by \cite[Theorem 2.3]{EHQ}. Thus, \[|M_{n,d,t}\setminus L_u|=\sum_{k=1}^{d}|M_{n-i_k,d-k+1,t}|=\sum_{k=1}^{d}{{n-i_k-(d-k)(t-1)}\choose {d-k+1}}=\]\[=\sum_{j=1}^{d}{{n-i_{d-j+1}-(j-1)(t-1)}\choose j}=\sum_{j=1}^{d}{a_j\choose j}.\]
\end{proof}

\begin{Remark}
In the previous lemma, let $r\geq1$ be the smallest integer for which $a_r>r-1$. Then the expansion $|M_{n,d,t}\setminus L_u|=\sum_{j=r}^{d}{a_j \choose j}$ is the binomial expansion of $|M_{n,d,t}\setminus L_u|$ with respect to $d$. Indeed, we have \[a_j-a_{j-1}=n-i_{d-j+1}-(j-1)(t-1)-n+i_{d-j+2}+(j-2)(t-1)=\]\[=i_{d-j+2}-i_{d-j+1}-t+1\geq t-t+1=1\text{ and }\]\[a_j=n-i_{d-j+1}-(j-1)(t-1)\geq (j-1)(t-1)\geq i_d-i_{d-j+1}-(j-1)t+j-1\geq j-1\]
for every $j$. Moreover, ${{a_j}\choose j}=0$ for $j<r$.
\end{Remark}

\begin{Definition}\label{power}
Let $n,d,t$ and $a$ be positive integers with $a\leq{{n-(d-1)(t-1)}\choose d}$. If $a={a_d \choose d}+{{a_{d-1}} \choose {d-1}}+\cdots + {a_r \choose r}$ is the binomial expansion of $a$ with respect to $d$, then we set $a_{r-1}=r-2$, $a_{d+1}=n-(d-1)(t-1)$ and $a_{d+2}=a_{d+1}+(t+1)$ and we define
\[a^{[d]_t}:=a^{[d]_t^k},\] where $k$ is the largest integer of the interval $[-1,d-r+1]$ with the property that $a_{d-k+1}-a_{d-k}\geq t+1$ and
\[a^{[d]_t^k}:=\sum_{j=d+1-k}^{d}{{a_j-(t-1)}\choose j+1}+{a_{d-k}-(2t-1)\choose {d-k+1}}+\sum_{j=r}^{d-k}{a_j\choose j}\]
for all $k\geq 0$ and
\[a^{[d]_t^{-1}}:={{n-d(t-1)}\choose{d+1}}.\]
\end{Definition}

Again for convenience, we set $0^{[d]_t}=0$ for positive integers $d$ and $t$.

\begin{Example}
We consider $n=28$, $d=8$, $t=3$ and $a=2018$ and we notice that $a=2018\leq 3003={{14}\choose{8}}={{n-(d-1)(t-1)}\choose d}$. Then the binomial expansion of $a$ with respect to $d$ is
$$2018={13 \choose 8}+{11 \choose 7}+{10 \choose 6}+{9 \choose 5}+{7 \choose 4}+{6 \choose 3}+{5 \choose 2}.$$
So, we obtain $a_1=0$, $a_2=5$, $a_3=6$, $a_4=7$, $a_5=9$, $a_6=10$, $a_7=11$, $a_8=13$, $a_9=14$ and $a_{10}=18$. Since $a_2-a_1=5\geq t+1=4$, the largest integer of the interval $[-1, 7]=[-1, 8-2+1]$, with the property that $a_{d-k+1}-a_{d-k}\geq t+1$, is $k=7$. Thus, we have $$a^{[d]_t}=2018^{[8]_3^7}=\sum_{j=2}^{8}{{a_j-2}\choose{j+1}}+{{a_1-5}\choose{2}}=$$$$={11 \choose 9}+{9 \choose 8}+{8 \choose 7}+{7 \choose 6}+{5 \choose 5}+{4 \choose 4}+{3 \choose 3}=82.$$
\end{Example}

The definition of the binomial operator $a\rightarrow a^{[d]_t}$ is justified by the following result.

\begin{Proposition}\label{cardshadow}
Let $\emptyset\neq L\subset M_{n,d,t}$ be a $t$-spread lex set with $a=|M_{n,d,t}\setminus L|$. Then
\[|M_{n,d+1,t}\setminus \shad_t(L)|=a^{[d]_t}.\]
\end{Proposition}
\begin{proof}
To begin with, we notice that $a<{{n-(d-1)(t-1)}\choose d}$, since $L\neq\emptyset$ and $|M_{n,d,t}|={{n-(d-1)(t-1)}\choose d}$ by \cite[Theorem 2.3]{EHQ}.

Let $a={a_d \choose d}+{{a_{d-1}} \choose {d-1}}+\cdots + {a_r \choose r}$ be the binomial expansion of $a$ with respect to $d$. Since $a<{{n-(d-1)(t-1)}\choose d}$, we have $a_d<n-(d-1)(t-1)$. Thus, we obtain $n-a_d-(d-1)(t-1)\geq 1$. Let us consider the $t$-spread monomial $u=x_{i_1}\cdots x_{i_d}$ where
\begin{enumerate}
  \item[(a)] $i_1:=n-a_d-(d-1)(t-1),i_2:=n-a_{d-1}-(d-2)(t-1),\ldots,i_k:=n-a_{d-k+1}-(d-k)(t-1),\ldots,i_{d-1}:=n-a_2-(t-1), i_d:=n-a_1$, if $r=1$;
  \item[(b)] $i_1:=n-a_d-(d-1)(t-1),i_2:=n-a_{d-1}-(d-2)(t-1),\ldots,i_{d-r+1}:=n-a_{r}-(r-1)(t-1),i_{d-r+2}:=n-(r-2)t,\ldots, i_d=n$, if $r>1$.
\end{enumerate}

By Lemma~\ref{cardLu}, $|M_{n,d,t}\setminus L_u|=a=|M_{n,d,t}\setminus L|$ and $L=L_u$, since $L$ and $L_u$ are $t$-spread lex sets.

According to Lemma~\ref{shad}, $\shad_t(L)$ is a $t$-spread lex set and with Lemma~\ref{cardLu}, we obtain the following four cases:
\begin{enumerate}
  \item If $r=1$ and $a_1\geq t$, then $n-i_d\geq t$. This follows that $\shad_t(L)=L_{ux_n}$ and \[|M_{n,d+1,t}\setminus \shad_t(L)|=|M_{n,d+1,t}\setminus L_{ux_n}|=\]\[=\sum_{j=2}^{d+1}{{n-i_{d+1-j+1}-(j-1)(t-1)}\choose j}=\sum_{j=2}^{d+1}{{a_{j-1}-(t-1)}\choose {j}}=\]\[=\sum_{j=1}^{d}{{a_j-(t-1)}\choose {j+1}}=a^{[d]_t^d}.\]
      Since $a_1-a_0=a_1+1\geq t+1$ and $k=d$ is the largest integer of the interval $[-1,d]$ with the property that $a_{d-k+1}-a_{d-k}\geq t+1$, we have $|M_{n,d+1,t}\setminus \shad_t(L)|=a^{[d]_t}$, as desired.
  \item If $(r\neq 1\text{ or }(r=1\text{ and }a_1<t))$ and there exists $k\in\{1,\ldots, d-r+1\}$ such that $a_{d-k+1}-a_{d-k}\geq t+1$, then we choose the largest $k\in [-1,d-r+1]$ with this property and we have $n-i_d<t$ and $i_{k+1}-i_k\geq 2t$. Thus,
      we obtain $\shad_t(L)=L_{ux_{i_{k+1}-t}}$ and \[|M_{n,d+1,t}\setminus \shad_t(L)|=|M_{n,d+1,t}\setminus L_{ux_{i_{k+1}-t}}| = \]\[= \sum_{j=1}^{d-k}{{n-i_{d-j+1}-(t-1)(j-1)}\choose{j}} + {{n-i_{k+1}-t-(t-1)(d-k)}\choose{d+1-k}} +\]\[+ \sum_{j=d+2-k}^{d+1}{{n-i_{d+1-j+1}-(t-1)(j-1)}\choose{j}} =\sum_{j=r}^{d-k}{{a_j}\choose{j}}+\]\[+{{a_{d-k}-(2t-1)}\choose{d-k+1}}+\sum_{j=d-k+1}^{d}{{a_j-(t-1)}\choose{j+1}}=a^{[d]_t^k}=a^{[d]_t}.\]
  \item If $(r\neq 1\text{ or }(r=1\text{ and }a_1<t))$, $a_{d-k+1}-a_{d-k}<t+1$ for all $k\in\{1,\ldots, d-r+1\}$ and $a_d\leq n-d(t-1)-2$, then $n-i_d<t$, $i_{k+1}-i_k<2t$ for all $k\in\{1,\ldots, d-1\}$ and $i_1-t\geq 1$. This implies that $\shad_t(L)=L_{ux_{i_1-t}}$ and \[|M_{n,d+1,t}\setminus \shad_t(L)|=|M_{n,d+1,t}\setminus L_{ux_{i_1-t}}|=\]\[=\sum_{j=1}^{d}{{n-i_{d-j+1}-(j-1)(t-1)}\choose j}+{{n-i_1-t-(t-1)d}\choose{d+1}}=\]\[=\sum_{j=r}^{d}{{a_j}\choose {j}}+{{a_d-(2t-1)}\choose{d+1}}=a^{[d]_t^0}.\] Since $a_{d-k+1}-a_{d-k}<t+1$ for all $k\in\{1,\ldots, d-r+1\}$ and $a_{d+1}-a_{d}=n-(d-1)(t-1)-a_{d}\geq n-(d-1)(t-1)-n-d(t-1)-2=t+1$, we obtain $|M_{n,d+1,t}\setminus \shad_t(L)|=a^{[d]_t^0}=a^{[d]_t}$, as desired.
  \item If $(r\neq 1\text{ or }(r=1\text{ and }a_1<t))$, $a_{d-k+1}-a_{d-k}<t+1$ for all $k\in\{1,\ldots, d-r+1\}$ and $a_d>n-d(t-1)-2$, then $n-i_d<t$, $i_{k+1}-i_k<2t$ for all $k\in\{1,\ldots, d-1\}$ and $i_1-t<1$. Therefore, $\shad_t(L)=\emptyset$ and \[|M_{n,d+1,t}\setminus \shad_t(L)|=|M_{n,d+1,t}|=a^{[d]_t^{-1}}=a^{[d]_t}\] because $a_{d-k+1}-a_{d-k}<t+1$ for all $k\in\{1,\ldots, d-r+1\}$, $a_{d+1}-a_d\leq t$ and $a_{d+2}-a_{d+1}=t+1$.
\end{enumerate}
\end{proof}

\begin{Remark}
If $L=\emptyset\subset M_{n,d,t}$, then $a=|M_{n,d,t}\setminus L|=|M_{n,d,t}|={{n-(d-1)(t-1)}\choose d}$ and $\shad_t(L)=\emptyset$. Moreover, we have $a_{d-1}=d-2$, $a_d=a_{d+1}=n-(d-1)(t-1)$, $a_{d+2}=a_{d+1}+(t+1)$ and $a^{[d]_t}=a^{[d]_t^k}$, where $k$ is the largest integer of the set $\{-1,0,1\}$ with the property that $a_{d-k+1}-a_{d-k}\geq t+1$. Since $a_{d+1}-a_d=0$, $k\in\{-1,1\}$ and
\begin{enumerate}
  \item $a^{[d]_t}={{n-d(t-1)}\choose{d+1}}$ if $k=-1$;
  \item $a^{[d]_t}={{a_d-(t-1)}\choose {d+1}}+{{a_{d-1}-(2t-1)}\choose d}={{n-d(t-1)}\choose {d+1}}$ if $k=1$.
\end{enumerate}
We notice that we get the same number in both cases and thus, $a^{[d]_t}={{n-d(t-1)}\choose{d+1}}=|M_{n,d+1,t}|=|M_{n,d+1,t}\setminus\shad_t(L)|$.
\end{Remark}

We now give the main result of this paper.

\begin{Theorem}\label{Kruskalkatona}
Let $f=(f(0),f(1),\ldots, f(d),\ldots)$ be a sequence of positive integers and $t\geq 1$ be an integer. The following conditions are equivalent:
  \begin{enumerate}
    \item[\emph{(1)}] there exists an integer $n\geq 0$ and a $t$-spread strongly stable ideal $$I\subset \KK[x_1,\ldots,x_n]$$ such that $f(d)=f_{t,d-1}(I)$ for all $d$.
    \item[\emph{(2)}] $f(0)=1$ and $f(d+1)\leq f(d)^{[d]_t}$ for all $d\geq 1$.
  \end{enumerate}
\end{Theorem}
\begin{proof}
According to Theorem~\ref{exists}, we may reduce to $t$-spread lex ideals instead of $t$-spread strongly stable ideals.

  For (1) $\Rightarrow$ (2), let $I\subset\mathbb{K}[x_1,\ldots,x_n]$ be a $t$-spread lex ideal with $f_{t,d-1}(I)=|M_{n,d,t}\setminus [I_d]_t|=f(d)\leq|M_{n,d,t}|={{n-(d-1)(t-1)}\choose d}$, where $[I_d]_t$ is the $t$-spread part of the $d$-th graded component of $I$. Since $[I_d]_t$ is a $t$-spread lex set in $M_{n,d,t}$ and $\shad_t([I_d]_t)\subset [I_{d+1}]_t$, it follows from Proposition~\ref{cardshadow} that \[f(d+1)=|M_{n,d+1,t}\setminus [I_{d+1}]_t|\leq |M_{n,d+1,t}\setminus\shad_t([I_d]_t)|=f(d)^{[d]_t}\] and of course we have $f(0)=f_{t,-1}(I)=1$.

  For (2) $\Rightarrow$ (1), let $n=f(1)$ and set $S=\mathbb{K}[x_1,\ldots, x_n]$. We first show by induction on $d$ that
  \[f(d)\leq {{n-(d-1)(t-1)}\choose{d}}.\] The assertion is trivial for $d=1$. Now assume that $f(d)\leq {{n-(d-1)(t-1)}\choose{d}}$ for some $d\geq 1$. Since $f(d+1)\leq f(d)^{[d]_t}$, it remains to prove that $f(d)^{[d]_t}\leq {{n-d(t-1)}\choose{d}}$.

  Let $f(d)={a_d \choose d}+{{a_{d-1}} \choose {d-1}}+\cdots + {{a_r} \choose r}$ be the binomial expansion of $f(d)$ with respect to $d$. Then $a_d\leq n-(d-1)(t-1)$.  By definition, $f(d)^{[d]_t}=f(d)^{[d]_t^k}$ where $k$ is the largest integer of the interval $[-1,d+r-1]$ with the property that $a_{d-k+1}-a_{d-k}\geq t+1$.

  If $k=-1$, then $f(d)^{[d]_t}={{n-d(t-1)}\choose{d+1}}$, as desired.

  If $k=0$, then $a_{d+1}-a_d\geq t+1$ and $a_d\leq n-d(t-1)-2$. Thus,
  \[f(d)^{[d]_t}=f(d)^{[d]_t^0}=\sum_{j=r}^{d}{{a_j}\choose{j}}+{{a_d-(2t-1)}\choose{d+1}}=\]
  \[=f(d)+{{a_d-(2t-1)}\choose{d+1}}<{{a_d+1}\choose{d}}+{{a_d+1-2t}\choose{d+1}}<\]\[<{{a_d+1}\choose{d}}+{{a_d+1}\choose{d+1}}={{a_d+2}\choose{d+1}}\leq
  {{n-d(t-1)}\choose{d+1}}.\]

  If $k\geq1$, then $a_{d-k+1}-a_{d-k}\geq t+1$ and $a_{d-k}+2\leq a_{d-k+1}-(t-1)$. Thus,
  \[f(d)^{[d]_t}=f(d)^{[d]_t^k}=\sum_{j=d+1-k}^{d}{{a_j-(t-1)}\choose{j+1}}+{{a_{d-k}-(2t-1)}\choose{d-k+1}}+\sum_{j=r}^{d-k}{{a_j}\choose{j}}=\]
  \[=\sum_{j=d+1-k}^{d}{{a_j-(t-1)}\choose{j+1}}+{{a_{d-k}-(2t-1)}\choose{d-k+1}}+a-\sum_{j=d-k+1}^{d}{{a_j}\choose{j}}<\]
  \[<\sum_{j=d+1-k}^{d}{{a_j-(t-1)}\choose{j+1}}+{{a_{d-k}+1}\choose{d-k+1}}+{{a_{d-k}+1}\choose{d-k}}=\]
  \[=\sum_{j=d+1-k}^{d}{{a_j-(t-1)}\choose{j+1}}+{{a_{d-k}+2}\choose{d-k+1}}\leq \] \[\leq\sum_{j=d+2-k}^{d}{{a_j-(t-1)}\choose{j+1}}+{{a_{d-k+1}-(t-1)}\choose{d-k+2}}+{{a_{d-k+1}-(t-1)}\choose{d-k+1}}=\]
  \[=\sum_{j=d+3-k}^{d}{{a_j-(t-1)}\choose{j+1}}+{{a_{d-k+2}-(t-1)}\choose{d-k+3}}+{{a_{d-k+1}-(t-1)+1}\choose{d-k+2}}\leq\ldots\leq\]
  \[\leq {{a_d-(t-1)}\choose{d+1}}+{{a_{d-1}-(t-1)+1}\choose{d}}\leq {{a_d-(t-1)+1}\choose{d+1}}.\]

  In the case that $f(d)<{{n-(d-1)(t-1)}\choose{d}}$, it follows by induction hypothesis that \[f(d)^{[d]_t}\leq {{a_d-(t-1)+1}\choose{d+1}}\leq {{n-(d-1)(t-1)-(t-1)}\choose{d+1}}={{n-d(t-1)}\choose{d+1}}.\]

  If $f(d)={{n-(d-1)(t-1)}\choose{d}}$, then \[f(d)^{[d]_t}={{a_d-(t-1)}\choose{d+1}}+{{a_{d-1}-(2t-1)}\choose{d}}=\]
  \[={{n-(t-1)(d-1)-(t-1)}\choose{d+1}}+{{d-2-(2t-1)}\choose{d}}={{n-d(t-1)}\choose{d+1}}.\]

  Hence we have $f(d)\leq {{n-(d-1)(t-1)}\choose{d}}=|M_{n,d,t}|$ for all $d$. In other words, $|M_{n,d,t}|-f(d)\geq 0$. Let $L_d$ be the unique $t$-spread lex set with $|L_d|=|M_{n,d,t}|-f(d)$.

  The existence of the $t$-spread lex ideal $I$ with $f_{t, d-1}(I)=f(d)=|M_{n,d,t}|-|L_d|$ is acquired in a similar way to the proof of the existence of $I^{\tlex}$ in Theorem~\ref{exists}. Thus, it remains to show that $\shad_t(L_d)\subset L_{d+1}$ for all $d$.

  Since $f(d+1)\leq f(d)^{[d]_t}$,
  \[|M_{n,d+1,t}\setminus L_{d+1}|=f(d+1)\leq f(d)^{[d]_t}=|M_{n,d,t}\setminus L_d|^{[d]_t}=|M_{n,d+1,t}\setminus \shad_t(L_d)|.\] It follows that $M_{n,d+1,t}\setminus L_{d+1}\subset M_{n,d+1,t}\setminus\shad_t(L_d)$ because $L_d$ is a $t$-spread lex set for every $d$. Thus, $\shad_t(L_d)\subset L_{d+1}$ for every $d$, as desired.
\end{proof}

\begin{Remark}\label{clasica}
For $t=1$ in Theorem~\ref{Kruskalkatona}, we obtain Theorem~\ref{clasicKK}. Indeed, let $a\text{ and }d\in\mathbb{Z}$. If $a={a_d \choose d}+{{a_{d-1}} \choose {d-1}}+\cdots + {a_r \choose r}$ is the binomial expansion of $a$ with respect to $d$, then $a_d>a_{d-1}>\ldots>a_r\geq r\geq1$. We set $a_{r-1}=r-2$. Since $a_{d-(d-r+1)+1}-a_{d-(d-r+1)}=a_r-a_{r-1}=a_r-r+2\geq 2=t+1$,
\[a^{[d]_1}=a^{[d]_1^{d-r+1}}=\sum_{j=r}^{d}{{a_j}\choose{j+1}}+{{a_{r-1}-1}\choose{r}}=\]\[=\sum_{j=r}^{d}{{a_j}\choose{j+1}}+{{r-2-1}\choose{r}}
=\sum_{j=r}^{d}{{a_j}\choose{j+1}}=a^{(d)}.\]
\end{Remark}

\begin{Example}
Let $f=(1,12,50,20,15,0,0,\ldots)$. If $t=1$, then $f$ satisfies the second condition of Theorem~\ref{Kruskalkatona} because we have
\begin{align*}
  f(1)^{[1]_1}=&  12^{(1)}={12\choose {1+1}}=66\geq f(2)=50,\\
  f(2)^{[2]_1}=&  50^{(2)}={10\choose {2+1}}+{5\choose{1+1}}=130 \geq f(3)=20,\\
  f(3)^{[3]_1}=&  20^{(3)}={6\choose {3+1}}=15\geq f(4)=15\text{ and }\\
  f(4)^{[4]_1}=& 15^{(4)}={6\choose {4+1}}=6\geq 0=f(5).
\end{align*}

We notice that $f(2)=50\leq f(1)^{[1]_t}={{f(1)-(t-1)}\choose{2}}={{13-t}\choose 2}$ if and only if $t=2$. Thus, $f$ does not satisfy the second condition of the previous theorem while $t>2$.

If $t=2$, then $f(3)=20\leq {{12-(3-1)}\choose{3}}=120$, $f(3)^{[3]_2}=f(3)^{[3]_2^1}={{6-(2-1)}\choose{3+1}}<{{6}\choose{4}}=15=f(4)$ and again we can not obtain a $2$-spread strongly stable ideal with the $f_2$-vector given by $f$.

In conclusion, $f=(1,12,50,20,15,0,0,\ldots)$ is the $f_t$-vector of a $t$-spread strongly stable ideal if and only if $t=1$.
\end{Example}

\end{document}